\documentclass[11pt]{amsart}

\usepackage{amssymb,mathrsfs,enumitem,amsmath,amsthm,bbold}
\usepackage[mathscr]{euscript}
\usepackage{color}

\usepackage[all]{xy}
\definecolor{Chocolat}{rgb}{0.36, 0.2, 0.09}
\definecolor{BleuTresFonce}{rgb}{0.215, 0.215, 0.36}
\definecolor{EgyptianBlue}{rgb}{0.06, 0.2, 0.65}

\usepackage[backend=bibtex,
    sorting=anyt,
    isbn=false,
    url=false,
    doi=false,
    maxbibnames=100
    ]{biblatex}
\usepackage{fourier}

\usepackage[colorlinks,final,hyperindex]{hyperref}
\hypersetup{citecolor=BleuTresFonce, urlcolor=EgyptianBlue, linkcolor=Chocolat}

\newtheorem{theorem}{Theorem}
\newtheorem{corollary}[theorem]{Corollary}

\theoremstyle{definition}

\newtheorem{remark}[theorem]{Remark}
\newtheorem{definition}[theorem]{Definition}

\DeclareMathAlphabet{\pazocal}{OMS}{zplm}{m}{n}

\DeclareMathOperator{\Der}{Der}

\DeclareMathAlphabet{\mathbbold}{U}{bbold}{m}{n}

\def\k{\mathbbold{k}}

\begin{document}

\title{Derived operations satisfy standard identities}

\author{Vladimir Dotsenko}
\address{Institut de Recherche Math\'ematique Avanc\'ee, Universit\'e de Strasbourg, 7 rue Ren\'e-Descartes, 67000 Strasbourg, France, and Institute of Mathematics and Mathematical Modeling, Pushkin St. 125, 050010 Almaty, Kazakhstan}

\email{vdotsenko@unistra.fr}

\begin{abstract}
A derived operation is a bilinear operation on a commutative associative algebra $A$ defined intrinsically out of its product and  several derivations of the product. We show that operators of left (or right) multiplications of a derived operation always satisfy a ``standard identity'' of certain order. In particular, it implies that each Rankin--Cohen bracket of modular forms, as well as each higher bracket of Kontsevich's universal deformation quantization formula for Poisson structures on $\mathbb{R}^n$, satisfies standard identities. 
\end{abstract}


\maketitle

\section{Introduction}

The goal of this short note is to record a simple but rather interesting example of how some remnants of algebraic properties of operations are ``inherited'' once building more complicated operations out of them. The algebraic structure we shall start with is an associative commutative algebra $A$ and a certain number of derivations $D_1$, \ldots, $D_n$ of that algebra. Equipped with this data, we may define a new bilinear operation $\{-,-\}$ on $A$ by the formula
 \[
\{a,b\}:=\sum_{j=1}^s f_j(a)g_j(b),   
 \] 
where $f_1$, $g_1$, \ldots, $f_s$, $g_s$ are linear combinations of iterations of the derivations $D_1$, \ldots, $D_n$. We shall refer to operations of this type as \emph{derived operations} on $A$. (It might be tempting to use the term ``derived bracket'', but we do not wish to create confusion with the derived brackets in the sense of \cite{MR2104437}.) Many interesting examples of nonassociative algebras in algebra, geometry, and mathematical physics are obtained via the derived operation construction (we shall recall a few of those below); in fact, there are even categorified derived operations in theoretical computer science \cite{fiore2014}. 

We shall demonstrate that, no matter what commutative associative algebra with several derivations we take, any derived operation satisfies an identity of a certain standard form: in a sense, the commutative associative law still shines through. 
Specifically, our main general result asserts that every derived operation satisfies some standard identities: there exists $d>0$ such that we have
 \[
s_{d,l}^{\{-,-\}}(a_1,\ldots,a_d)=s_{d,r}^{\{-,-\}}(a_1,\ldots,a_d)=0     
 \]
for all $a_1,\ldots,a_d\in A$. Here $s_{d,l}^{\{-,-\}}(a_1,\ldots,a_d)$ and $s_{d,r}^{\{-,-\}}(a_1,\ldots,a_d)$ are versions of the so called standard polynomials that play an important role in the context of associative PI-algebras, as well as of identities of Lie algebras: 
\begin{gather*}
s_{d,l}^{\{-,-\}}(a_1,\ldots,a_d):=\sum_{\sigma\in S_{d-1}}(-1)^{\sigma}\{a_{\sigma(1)},\{a_{\sigma(2)},\{\ldots, \{a_{\sigma(d-2)}, \{a_{\sigma(d-1)},a_d\}\} \ldots\}\}\},\\ 
s_{d,r}^{\{-,-\}}(a_1,\ldots,a_d):=\sum_{\sigma\in S_{d-1}}(-1)^{\sigma}\{\{\{\ldots  \{\{a_d,a_{\sigma(d-1)}\},a_{\sigma(d-2)}\}, \ldots\},a_{\sigma(2)} \}, a_{\sigma(1)}\}.
\end{gather*}
The bound on $d$ is an explicit expression depending on the derived operation; we obtain a very rough bound in the general case, and show how to substantially improve it in the case where the derivations $D_1$, \ldots, $D_n$ generate a finite-dimensional Lie subalgebra of $\Der(A)$. 

In Section \ref{sec:examples}, we discuss several examples of derived operations in order to highlight the scope of applicability of our results. Then, in Section \ref{sec:result}, we explain the proofs of the two main results of the paper, and spell out what exactly these results mean for Rankin--Cohen brackets of modular forms, the derived operations that served as our original motivation. Finally, in Section \ref{sec:higher}, we note a generalization of our result for operations with a higher number of arguments.

\section{Examples of derived operations}\label{sec:examples}

\subsection{Finite-dimensional Lie algebras of derivations}

\subsubsection{Derived operations in classical mechanics and the theory of PDEs}

It is perhaps fair to say that one of the most famous non-associative operations is the Poisson bracket of smooth functions on $\mathbb{R}^{2n}$ given by
 \[
\{f,g\}:=\sum_{i=1}^n\left(X_i(f)Y_i(g)-X_i(g)Y_i(f)\right) ,    
 \]
where $X_i=\frac{\partial\phantom{f}}{\partial x_i}$ and $Y_i=\frac{\partial\phantom{f}}{\partial x_{n+i}}$. Of course, it is a derived operation for the abelian Lie algebra spanned by the derivations $X_i$ acting on the commutative associative algebra of smooth functions; it admits an obvious generalization where $\mathbb{R}^{2n}$ is replaced by a symplectic manifold. If instead of $\mathbb{R}^{2n}$, one considers $\mathbb{R}^{2n+1}$, or, more generally, a contact manifold, there is an important operation $\{f,g\}_J$ going back to the work of Jacobi \cite{MR1579190} defined by 
 \[
\{f,g\}_J:=fZ(g)-gZ(f)+\sum_{i=1}^n\left(X_i(f)Y_i(g)-X_i(g)Y_i(f)\right),     
 \]
where $X_i=\frac{\partial\phantom{f}}{\partial x_i}$, $Y_i=\frac{\partial\phantom{f}}{\partial x_{n+i}}+x_i\frac{\partial\phantom{f}}{\partial x_{2n+1}}$, and $Z=\frac{\partial\phantom{f}}{\partial x_{2n+1}}$. 
Note that the derivations $X_i, Y_i, Z$ have the commutation relations of the Heisenberg Lie algebra:
 \[
[X_i,X_j]=[Y_i,Y_j]=[X_i,Z]=[Y_i,Z]=0, \quad [X_i,Y_j]=\delta_{ij}Z.   
 \]
We refer the reader to the work of Kirillov \cite{MR438390} and Lichnerowicz \cite{MR524629} for more information about the geometric meaning of these Jacobi brackets. Moreover, Kirillov \cite{MR1029493} proved that these brackets satisfy standard identities of certain degrees depending on $n$ (if one interprets these Lie algebras as the Lie algebra of Hamiltonian and contact vector fields respectively, this follows from the general fact that Lie algebras of vector fields on finite-dimensional manifolds satisfy nontrivial Lie identities \cite[Sec.~2.1.3]{MR4368864}, but the lowest  degree of a nontrivial identity is slightly lower for Hamiltonian and contact vector fields, as shown by Kirillov).

In the context of partial differential equations, Mayer \cite{MR1509865} defined an operation $\{f,g\}_M$ on smooth functions on $\mathbb{R}^{2n+1}$ closely related to the Jacobi bracket; it is given by 
\[
\{f,g\}_M:=\sum_{i=1}^n\left(X_i(f)Y_i(g)-X_i(g)Y_i(f)\right) ,    
 \]
where $X_i, Y_i, Z$ are as above. (The Jacobi bracket satisfies the Jacobi identity but does not satisfy the Leibniz rule with respect to the product of functions, while the Mayer bracket satisfies the Leibniz rule but does not satisfy the Jacobi identity, though one can find weaker identities for Mayer brackets \cite{MR2514853,MR1867927}.) 

\subsubsection{Novikov algebras and related examples} The following construction first appeared in \cite{GD79}, where it is attributed to S.~I.~Gelfand. Let $A$ be a commutative associative algebra, and let $\partial$ be a derivation of that algebra. One may consider the new binary operation on $A$ given by $a\circ b:=a\partial(b)$. This operation satisfies the identities
\begin{gather*}
(a\circ b)\circ c-a\circ(b\circ c)=(b\circ a)\circ c - b\circ(a\circ c),\\
(a\circ b)\circ c=(a\circ c)\circ b,
\end{gather*}
which define the class of nonassociative algebras called Novikov algebras (named by Osborn \cite{MR1163779} as a recognition of \cite{MR802121}); in fact, every Novikov algebra is a subalgebra of an algebra of this type, see \cite{MR3590862}. 

By the very definition, the Novikov product is an instance of a derived operation for the one-dimensional abelian Lie algebra spanned by $\partial$, making it arguably the simplest example of a derived operation. In this case, standard identities are well known: indeed, the second of Novikov  identities, once written as
 \[
(a_3\circ a_1)\circ a_2-  (a_3\circ a_2)\circ a_1=0, 
 \]
is easily seen to be the standard identity $s_{3,r}^{-\circ-}(a_1,a_2,a_3)=0$; furthermore, it is known \cite{MR1743658} that the identity 
$s_{4,l}^{-\circ-}(a_1,a_2,a_3,a_4)=0$ holds in every Novikov algebra.

We remark that we can also consider the operations $a\circ b-b\circ a$ and $a\circ b+b\circ a$ defined on any Novikov algebra. In terms of commutative associative algebra with derivations, we have 
\begin{gather*}
[a,b]:=a\circ b-b\circ a=a\partial(b)-b\partial(a),\\
a\star b:=a\circ b+b\circ a=a\partial(b)+b\partial(a),
\end{gather*}
so these are derived operations. Standard identities for them are also known: it is well known \cite[Sec.~2.1.3]{MR4368864} that 
$s_{5,l}^{[-,-]}(a_1,a_2,a_3,a_4,a_5)=0$, and one can show that $s_{4,l}^{-\star}(a_1,a_2,a_3,a_4)=0$, see \cite[Prop.~2.3]{MR1944285}.  

\subsubsection{Rankin--Cohen algebras}\label{sec:rc}

In \cite{MR1280058}, Zagier, motivated by a particular series of bilinear differential operators acting on modular forms, defined an algebraic structure that he called Rankin--Cohen algebras.

\begin{definition}
Let $A=\bigoplus_{n\ge 0}A_n$ be a graded commutative associative algebra and let $D$ be a derivation of degree $2$, so that $D(A_n)\subset A_{n+2}$. The \emph{Rankin--Cohen bracket} $[a,b]_n$ on $A$ is the unique bilinear operation that is given, for $a\in A_k$ and $b\in A_l$, by the formula 
\begin{equation}\label{eq:RC-weight}
[a,b]_n:=\sum_{r+s=n} (-1)^r\binom{n+k-1}{s}\binom{n+l-1}{r}D^r(a)D^s(b). 
\end{equation}   
\end{definition}
The operations $[-,-]_n$ have obvious symmetry properties 
 \[
[a,b]_n=(-1)^n[b,a].     
 \]
Note that the definition of these operations makes use of the degrees of $a,b$, so at a first glance they do not belong in the context we consider. However, one may argue as follows. Let us denote by $W$ the derivation of $A$ defined by $W(a)=ka$ for $a\in A_k$. Clearly, we have $[W,D]=2D$. Note that we may write 
\begin{equation}\label{eq:RC-Ug}
[a,b]_n:=\sum_{r+s=n} (-1)^rD^r\binom{W+n-1}{s}(a)D^s\binom{W+n-1}{r}(b), 
\end{equation}
where we use the fact that for $k\ge 0$ the binomial coefficient $\binom{z}{k}$ is a polynomial in~$z$, and we may substitute any element of any associative algebra instead of~$z$. Formula \eqref{eq:RC-Ug} defines the \emph{generalized Rankin--Cohen brackets} $[-,-]_n$ on any associative commutative algebra $A$ that is a module over the two-dimensional Lie algebra~$L$ spanned by $W$ and $D$, where $L$ acts by derivations of the product of~$A$; thus, each Rankin--Cohen bracket is a derived operation. 

It interesting to note that if one considers all Rankin--Cohen brackets together and allows coefficients depending on degrees of arguments, Labriet and Poulain d'Andecy have shown in \cite{MR4691919} that all identities follow from identities with three arguments. Establishing standard identities for each individual $[-,-]_n$ from that point of view seems to be a highly nontrivial task. In particular, the coefficients of identities of \cite{MR4691919} depend on weights of the arguments, or, in other words, involve the derivation $W$, while the standard identities we exhibit are linear combinations with constant coefficients of iterations of $[-,-]_n$.

\subsubsection{Derived operations for one-dimensional subalgebras of~\texorpdfstring{$\Der(A)$}{DerA}}

In this section, we recall a number of examples of derived operations for one-dimensional subalgebras of $\Der(A)$ discovered by Dzhumadildaev in his studies of identities of particular form. 

It is proved in \cite[Sec.~5--9]{MR2514853} that for any derivation $\partial$ of a commutative associative algebra $A$, the brackets
 \[
\{a,b\}=a\partial^2(b)-b\partial^2(a) \text{ and  } \{a,b\}=\partial(a)\partial^2(b)-\partial(b)\partial^2(a)   
 \]
satisfy a fully skew-symmetric identity of degree $4$, and the bracket 
 \[
\partial^3(a)b-2\partial^2(a)\partial(b)+2\partial(a)\partial^2(b)-a\partial^3(b)   
 \]
satisfies a fully skew-symmetric identity of degree $5$. 


In \cite[Sec.~6.6]{MR2676255}, the product 
 \[
a\star b=\partial^3(a)b+4\partial^2(a)\partial(b)+5\partial(a)\partial^2(b)+2a\partial^3(b)   
 \]
on a commutative associative algebra $A$ with a derivation $\partial$ is considered; it is shown that if $A=\k[x]$ and $\partial=\frac{\partial\phantom{x}}{\partial x}$, then $(A,\star)$ is a simple algebra satisfying the so-called $0$-Alia identity, and that this algebra is ``exceptional'' in a certain sense. 

Furthermore, in \cite[Sec.~7]{MR2676255}, it is proved that the products 
\begin{gather*}
a\star_1 b=\partial(a)\partial^2(b),\\
a\star_2 b=-a\partial^m(b)+\partial^m(a)b+\partial^m(ab)
\end{gather*}
(the second defined for any $m\in\mathbb{N}$) on a commutative associative algebra $A$ with a derivation $\partial$ satisfy the $1$-Alia identity. 
Interestingly enough, it is indicated in \cite[Th.~7.3]{MR2676255} that the first of these products satisfies the standard identity
$s_{4,r}^{-\star_1-}(a_1,a_2,a_3,a_4)=0$, whereas our result would only predict a standard identity of degree $9$. 

Finally, in \cite[Sec.~8]{MR2676255}, the product
 \[
a\star b=\partial(\partial(a)b)   
 \]
on a commutative associative algebra $A$ with a derivation $\partial$ is considered; it is proved that if $A=\k[x]$ and $\partial=\frac{\partial\phantom{x}}{\partial x}$, then $(A,\star)$ is a simple $1$-Alia algebra. 
 
\subsection{More general finite sets of derivations}

\subsubsection{``Almost Poisson'' brackets defined by a bivector field}

A class of algebras known as ``almost Poisson algebras'' \cite{MR1747916} arises from considering brackets  
 \[
\{f,g\}=\omega(df,dg)      
 \]
of smooth functions on a manifold, where $\omega$ is an arbitrary bivector field. This means that in local coordinates $x_1$, \ldots, $x_n$, this bracket has the form
 \[
\{f,g\}=\sum_{i<j}\omega^{ij}\left(\frac{\partial f}{\partial x_i}\frac{\partial g}{\partial x_j}-\frac{\partial g}{\partial x_i}\frac{\partial f}{\partial x_j}\right).     
 \]
Such brackets are derived operations. The easiest way to understand this is to write 
 \[
\omega= \sum_{i<j}\omega^{ij}\frac{\partial\phantom{f}}{\partial x_i}\wedge \frac{\partial\phantom{f}}{\partial x_j}=
\sum_{i=1}^{n-1} \frac{\partial\phantom{f}}{\partial x_i}\wedge\left(\sum_{j=i+1}^n\omega^{ij}\frac{\partial\phantom{f}}{\partial x_j}\right),
 \]
which explicitly identifies $2(n-1)$ derivations of which the bracket $\{-,-\}$ is made. 

\subsubsection{Higher brackets of the universal deformation quantization formulas}

The universal deformation quantization formula of Kontsevich \cite{MR2062626} produces, for every smooth Poisson manifold $(M,\pi)$, where $\pi$ is a bivector field on $M$ satisfying the Maurer--Cartan equation $[\pi,\pi]_{SN}=0$ for the so-called Schouten--Nijenhuis bracket of multivector fields, an $\mathbb{R}[[\hbar]]$-linear star product on $C^\infty(M)[[\hbar]]$ of the form
 \[
f\star g=\sum_{k\ge 0}B_k(f,g)\hbar^k,      
 \]
where 
 \[
B_k(f,g)=\sum_{\Gamma} w_\Gamma B_{\Gamma,\pi}(f,g)     
 \]
is the sum, with certain constant weights $w_\Gamma$, of certain bi-differential operators $B_{\Gamma,\pi}(f,g)$ corresponding to directed graphs $\Gamma$ of a particular type. Each such graph has with $2$ vertices ``of the first type'' labelled $f,g$, $k$ vertices ``of the second type'' labelled $1$, \ldots, $k$, and $2k$ directed edges $e_1,f_1$, \ldots, $e_k$, $f_k$, such that for each $j=1,\ldots,k$, the source of both $e_j$ and $f_j$ is the vertex labelled $j$; these graphs are not allowed to have either double edges or directed cycles. The bi-differential operator corresponding to a graph $\Gamma$ puts $\alpha$ at each of the vertices $1$, \ldots, $k$, and uses the arrows as instructions where to apply the partial derivatives in 
 \[
\pi= \sum_{i<j}\pi^{ij}\frac{\partial\phantom{f}}{\partial x_i}\wedge \frac{\partial\phantom{f}}{\partial x_j} .    
 \]
Arguing as in the previous section, we see that each $B_k(f,g)$ is a derived operation for a finitely generated Lie subalgebra of $\Der(C^\infty(M))$.

\subsubsection{An example of Dzhumadildaev}

In \cite{MR2514853}, Dzhumadildaev studies anticommutative nonassociative algebras satisfying fully skew-symmetric identities. In particular, in \cite[Th.~5.3]{MR2514853} he shows that for any two derivations $\partial_1$ and $\partial_2$ on a commutative associative algebra $A$, the bracket 
 \[
\{a,b\}=\partial_1(a)\partial_2(b)-\partial_1(b)\partial_2(a)   
 \]
satisfies a fully skew-symmetric identity of degree $4$; our Theorem \ref{th:FG} below implies that it also satisfies a standard identity of degree $16$.

\section{The main results}\label{sec:result}

In this section we shall prove two general results mentioned in the introduction. Throughout this section, we consider an arbitrary derived operation obtained using linear combinations of iterations of finitely many derivations $D_1$, \ldots, $D_n$, and denote by $\mathfrak{g}$ the Lie subalgebra of $\Der(A)$ generated by these derivations. The first of our results asserts that every derived operation satisfies standard identities, and the second provides a better bound for the degree of such identities in the case where our derivations generate a finite-dimensional Lie subalgebra of $\Der(A)$. 

For the first of our results, let us consider the universal enveloping algebra $U(\mathfrak{g})$, which is also generated by $\{D_1,\ldots,D_n\}$, and define on it a filtration $G^\bullet U(\mathfrak{g})$ that is uniquely determined by the conditions 
 \[
G^0 U(\mathfrak{g})=\k 1, \quad G^1 U(\mathfrak{g})=\k1\oplus \k\{D_1,\ldots,D_n\}, \quad G^k U(\mathfrak{g})G^l U(\mathfrak{g})\subset G^{k+l} U(\mathfrak{g}).   
 \]
Suppose that our derived operation is given by the formula
 \[
\{a,b\}:=\sum_{j=1}^s f_j(a)\cdot g_j(b),    
 \]
where $f_j\subset G^{n_j}U(\mathfrak{g})$,  $g_j\subset G^{m_j}U(\mathfrak{g})$. We define the \emph{$D$-order} of $\{-,-\}$ to be given by $\max_j(n_j+m_j)$.
In the following theorem, we assume $n>1$; the theorem after that covers the case of a finite-dimensional Lie algebra $\mathfrak{g}$, and so includes the case $n=1$.

\begin{theorem}\label{th:FG}
Every derived operation $\{-,-\}$ defined using $n>1$ derivations $D_1$, \ldots, $D_n$ satisfies a standard identity: there exists $d>0$ such that we have
 \[
s_{d,l}^{\{-,-\}}(a_1,\ldots,a_d)=s_{d,r}^{\{-,-\}}(a_1,\ldots,a_d)=0     
 \]
for all $a_1,\ldots,a_d\in A$. In fact, we can take $d=1+\frac{n^{p+1}-1}{n-1}$ for any $p\ge m+1$, where $m$ is the $D$-order of $\{-,-\}$.
\end{theorem} 

\begin{proof}
It is enough to consider the left standard identity $s_{d,l}^{\{-,-\}}=0$, since 
 \[
s_{d,r}^{\{-,-\}}=s_{d,l}^{\{-,-\}^{op}}   
 \]
with $\{a,b\}^{op}:=\{b,a\}$. The iteration 
 \[
\{a_{1},\{a_{2},\{\ldots, \{a_{d-2}, \{a_{d-1},a_d\}\} \ldots\}\}\}   
 \] 
of our operation can be written as the sum of terms $h_1(a_1) h_2(a_2)\cdots h_d(a_d)$,
where $h_j\in G^{k_j} U(\mathfrak{g})$, and $k_1+\cdots+k_d\le (d-1)m$. Now let us take $p\in\mathbb{N}$, and choose $d-1=1+n+n^2+\cdots+n^p=\frac{n^{p+1}-1}{n-1}$ to be the number of all noncommutative monomials in $n$ variables $x_1$, \ldots, $x_n$ of total degree at most $p$. We wish to show that for $p$ sufficiently large, the standard identity of degree $d$ is satisfied for our derived operation. We may expand elements $h_1$, \ldots, $h_d$, and assume that each of them is a monomial in the free algebra $\k\langle X\rangle$, of which $U(\mathfrak{g})$ is a quotient. Since we antisymmetrize over $1,\ldots,d-1$, in order for the result to be nonzero, among the first $d-1$ elements there can be at most $1$ monomial of degree $0$, at most $n$ monomials of degree $1$, at most $n^2$ monomials of degree $2$, \ldots, at most $n^p$ monomials of degree $p$. The sum of degrees of these monomials is therefore at least
 \[
n+2n^2+\cdots+pn^p=\frac{pn^{p+1}-n-n^2-\cdots-n^p}{n-1}.
 \]
Thus, we have an inequality
 \[
\frac{pn^{p+1}-n-n^2-\cdots-n^p}{n-1}\le (d-1)m= \frac{n^{p+1}-1}{n-1}m. 
 \] 
Multiplying by $n-1$ and reorganizing the terms, we get
 \[
\left(p-m-\frac{1}{n-1}\right)(n^{p+1}-1)+p+1\le 0, 
 \] 
which is clearly false for $p\ge m+1$. 
\end{proof}

The next result is similar, but uses a different notion of order of the given derived operation $\{-,-\}$, corresponding to the
standard filtration $F^\bullet U(\mathfrak{g})$ uniquely determined by the conditions 
 \[
F^0 U(\mathfrak{g})=\k 1, \quad F^1 U(\mathfrak{g})=\k1\oplus \mathfrak{g}, \quad F^k U(\mathfrak{g})F^l U(\mathfrak{g})\subset F^{k+l} U(\mathfrak{g}).  
 \]
Suppose that our derived operation is given by the formula
 \[
\{a,b\}:=\sum_{j=1}^s f_j(a)\cdot g_j(b),    
 \]
where $f_j\subset F^{n_j}U(\mathfrak{g})$,  $g_j\subset F^{m_j}U(\mathfrak{g})$. We define the $\mathfrak{g}$-order of $\{-,-\}$ to be given by $\max_j(n_j+m_j)$.

\begin{theorem}\label{th:FD}
Suppose that the Lie algebra $\mathfrak{g}$ is finite-dimensional. Then every derived operation satisfies a standard identity: there exists $d>0$ such that we have
 \[
s_{d,l}^{\{-,-\}}(a_1,\ldots,a_d)=s_{d,r}^{\{-,-\}}(a_1,\ldots,a_d)=0     
 \]
for all $a_1,\ldots,a_d\in A$. In fact, we can take $d=1+\binom{\dim\mathfrak{g}+p}{\dim\mathfrak{g}}$ for any $p> m\left(1+\frac1{\dim\mathfrak{g}}\right)$, where $m$ is the $\mathfrak{g}$-order of $\{-,-\}$.
\end{theorem} 

\begin{proof}
As above, it is enough to consider the left standard identity $s_{d,l}^{\{-,-\}}=0$.
 The iteration 
 \[
\{a_{1},\{a_{2},\{\ldots, \{a_{d-2}, \{a_{d-1},a_d\}\} \ldots\}\}\}   
 \] 
of our operation can be written as the sum of terms $h_1(a_1) h_2(a_2)\cdots h_d(a_d)$,
where $h_j(a_j)\in F^{k_j} U(\mathfrak{g})$, and $k_1+\cdots+k_d\le (d-1)m$. Denote for brevity $n:=\dim\mathfrak{g}$. Now let us take $p\in\mathbb{N}$, and choose $d-1=\binom{n+p}{n}$ to be the number of all commutative monomials in $n$ variables of total degree at most $p$. We wish to show that for $p$ sufficiently large, the standard identity of degree $d$ is satisfied for our derived operation. We may expand elements $h_1$, \ldots, $h_d$, and assume that each of them is a monomial in $S(\mathfrak{g})\cong U(\mathfrak{g})$. Since we antisymmetrize over $1,\ldots,d-1$, in order for the result to be nonzero, among the first $d-1$ elements there can be at most $1$ monomial of degree $0$, at most $n$ monomials of degree $1$, at most $\binom{n+1}2$ monomials of degree $2$, \ldots, at most $\binom{n+p-1}p$ monomials of degree $p$. The sum of degrees of these monomials is therefore at least
 \[
n+2\binom{n+1}2+\cdots+p\binom{n+p-1}{p}= n\binom{n+p}{p-1}.
 \]
Thus, we have an inequality
 \[
n\binom{n+p}{p-1}\le (d-1)m= \binom{n+p}{n}m. 
 \] 
Simplifying, we get
 \[
p\le m \left(1+\frac1n\right),  
 \] 
so to ensure that this inequality does not hold, it is enough to take $p>m \left(1+\frac1n\right)$, and $d=1+\binom{n+p}{n}$ to complete the proof. 
\end{proof}

Slightly refining the proof of this theorem, one obtains standard identities for Rankin--Cohen brackets, which was in fact the starting point of this paper. As indicated in Section \ref{sec:rc}, many identities for Rankin--Cohen brackets depend on weights of the arguments, so that the known identity
$m[[a,b]_1,c]_0+k[[b,c]_1,a]_0+l[[c,a]_1,b]_0=0$ for all $a\in A_k$, $b\in A_l$, $c\in A_m$ means in our terms
 \[
[[a,b]_1,W(c)]_0+[[b,c]_1,W(a)]_0+[[c,a]_1,W(b)]_0=0.     
 \]
Thus, outside the context of the present paper, the following result is somewhat surprising.

\begin{corollary}
The generalized Rankin--Cohen bracket $[-,-]_n$ satisfies the standard identity of degree $d=\frac{9n(n+1)}{2}-1$.
\end{corollary} 

\begin{proof}
Let us note that the formula 
 \[
[a,b]_n:=\sum_{r+s=n} (-1)^rD^r\binom{W+n-1}{s}(a)D^s\binom{W+n-1}{r}(b)     
 \]
implies that in this case the elements $f_j$, $g_j$ in the formula for the derived operation belong to the augmentation ideal of the universal enveloping algebra of the two-dimensional Lie algebra generated by $D$ and $W$. Therefore in the above proof for the bracket $[-,-]_n$ $\mathfrak{g}$-order $2n$, the sum of filtrations of the polynomials $h_1$, \ldots, $h_{d-1}$ is strictly less than $2n(d-1)$. Note that we have 
 \[
d=\frac{9n(n+1)}{2}-1=\frac{3n(3n+1)}{2}+3n-1=1+2+3+\cdots+(3n-1)+(3n)+(3n-1).
 \] 
Since we antisymmetrize elements from the augmentation ideal over $1,\ldots,d-1$, there can be at most two monomials of degree $1$, at most three monomials of degree $2$, etc., and we have $d-1=2+3+\cdots+(3n-1)+(3n)+(3n-1)$ monomials, hence the sum of their degrees is at least 
 \[
2+2\cdot 3+\cdots+(3n-1)(3n)+(3n-1)(3n)=\frac{(3n-1)(3n)(3n+1)}3+(3n-1)(3n)     
 \]
Thus, we have the inequality 
 \[
\frac{(3n-1)(3n)(3n+1)}3+(3n-1)(3n)<2n(d-1)=2n\left(\frac{(3n)(3n+1)}{2}+3n-1-1\right),    
 \]  
which is a contradiction since the left hand side is easily seen equal to the right hand side.  
\end{proof}

\section{A generalization for higher number of arguments}\label{sec:higher}

Let us record an obvious generalization of our main results. Suppose that, like above, $A$ is an associative commutative algebra, and $D_1$, \ldots, $D_n$ are derivations of that algebra. We may define a $m$-linear operation $\{-,\ldots,-\}$ on $A$ by the formula
 \[
\{a_1,\ldots,a_k\}:=\sum_{j=1}^s f_j^{(1)}(a_1)\cdots f_j^{(k)}(a_k),   
 \] 
where $f_j^{(i)}$, $i=1,\ldots,k$, $j=1,\ldots,s$, are linear combinations of iterations of the derivations $D_1$, \ldots, $D_n$. Then
\begin{multline}\label{eq:standard-k}
\sum_{\sigma\in S_{(k-1)d}}(-1)^{\sigma}\{a_{\sigma(1)},a_{\sigma(2)},\ldots,a_{\sigma(k-1)},\{a_{\sigma(k)},\ldots,a_{\sigma(2k-2)}\{\ldots,\\ 
 \{a_{\sigma((d-1)(k-1)+1)}, \ldots a_{\sigma(d(k-1)))},a_{d(k-1)+1}\} \ldots\}\}\}=0.     
\end{multline}
The proof is completely analogous to that of our main result. This for example applies to the Wronskian-type operations that appear in the literature, see, e.g., \cite{MR2169395,MR2137432,kiselev2025jacobiidentitieswronskiandeterminants,kiselev2025wronskiansnarybracketsfinitedimensional}. Identities of the form \eqref{eq:standard-k} are known for some of such operations since they are ``homotopy Lie algebras with operations of wrong homotopical degrees'', see \cite[Sec.~6]{MR4151721} for precise statements. 


\section*{Acknowledgements. } I thank Don Zagier for the most illuminating talk on Rankin--Cohen brackets at the Manin Memorial Conference at the Max-Planck Institute for Mathematics in August 2025. I also gratefully acknowledge conversations with Gleb Koshevoy and Vadim Schechtman at that conference, as well as the extraordinary effort the organizers of the conference and the staff of the MPIM put into making this conference a truly remarkable event, and a fitting tribute to Yuri Ivanovich Manin. Finally, I am grateful to Askar Dzhumadildaev for interesting discussions of a preliminary version of this work, and to Arthemy Kiselev for bringing \cite{kiselev2025jacobiidentitieswronskiandeterminants,kiselev2025wronskiansnarybracketsfinitedimensional} to my attention.

\section*{Funding. }
This research was funded by the  Science Committee of the Ministry of Science and Higher Education of the Republic of Kazakhstan (Grant No. BR 28713025). 

\section*{Competing interests. } The author has no competing interests to declare. 

\printbibliography

\end{document}